\documentclass{amsart}
\usepackage{color}

\usepackage{tipa}
\usepackage{amssymb}
\usepackage{stmaryrd}
\usepackage{graphicx}
\usepackage{amsmath,amsfonts}
\usepackage{mathtools}
\usepackage{booktabs}
\usepackage{tikz}
\usepackage{setspace}
\usepackage{float}
\usetikzlibrary{arrows,shapes,chains}
\usepackage{array}
\newcolumntype{C}[1]{>{\centering}p{#1}}
\usetikzlibrary{calc}
\usetikzlibrary{intersections,through}

\newtheorem{theorem}{Theorem}[section]
\newtheorem{lemma}[theorem]{Lemma}

\theoremstyle{definition}
\newtheorem{definition}[theorem]{Definition}
\newtheorem{example}[theorem]{Example}

\newtheorem{corollary}[theorem]{Corollary}
\newtheorem{remark}[theorem]{Remark}
\newtheorem{question}[theorem]{Question}

\numberwithin{equation}{section}

\begin{document}

\title[A new perspective of arithmetic billiards]{A new perspective of arithmetic billiards}

\author{Yangcheng Li}
\address{School of Mathematical Sciences, South China Normal University, Guangzhou 510631, People's Republic of China}
\email{liyangchengmlx@163.com}


\date{}


\begin{abstract}
We study the problem of arithmetic billiards from a new perspective. We first raise a similar problem about reflecting lights inside grids. For the solution to this problem, we will give three proofs. Next, we consider a similar problem in plane grids and give its solution. Moreover, we extend this two problems to $p$-dimensional space, where $p\geq2$. In this process, we introduce two mappings about finite discrete sets, and get two finite abelian groups. In addition, we give the definition of circular sequences and consider some combinatorial properties of circular sequences.
\end{abstract}

\maketitle

\section{Introduction}
In 1978, Rauch \cite{Rauch} studied the problem of how many light sources are needed to illuminate the interior of a bounded domain. He considered the problem from the perspective of geometry, using some simple geometric properties of the ellipse and an elementary compactness argument. In this paper, we investigate a similar problem in a plane grid. Our tools are methods in number theory and algebra. It is worth mentioning that Perucca, De Sousa, and Tronto \cite{Perucca} studied the problem of arithmetic billiards and got some interesting results.

Let us describe this problem first. A plane grid consists of several horizontal and vertical line segments (as shown in Figure 1).
\begin{figure}[H]
\begin{tikzpicture}[scale=0.9]
\draw[line width=0.6pt] (0,0)--(6,0)--(6,4)--(0,4)--(0,0);
\draw[line width=0.35pt] (0,1)--(6,1) (0,2)--(6,2) (0,3)--(6,3) (1,0)--(1,4) (2,0)--(2,4) (3,0)--(3,4) (4,0)--(4,4) (5,0)--(5,4);
\end{tikzpicture}
\vspace{0.2cm}
\caption{A $6\times4$ plane grid}
\end{figure}
\vspace{-0.6cm}
We consider the problem of how to connect all points in the plane grid with lines. We assume that the boundary of the plane grid is composed of mirrors, which means that light can be reflected. If a beam of light enters the plane grid from the boundary and reflects when it meets the boundary, then the beam of light will pass through several points in the plane grid. We want to know at least how many beams of light are needed to pass through all the points in this plane grid. Therefore, we raise a problem as follow.
\begin{question}
A $m\times n$ plane grid whose boundary is composed of mirrors, how many beams of light can pass through all points in the plane grid at least?
\end{question}

\begin{example}
The following is an example of a $6\times4$ plane grid, and this plane grid needs at least $3$ beams of light to pass through all points.
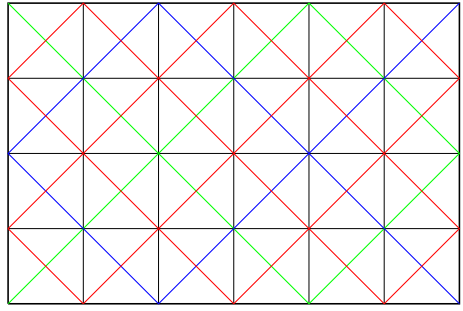
\begin{figure}[H]
\begin{tikzpicture}[scale=0.9]
\draw[line width=0.6pt] (0,0)--(6,0)--(6,4)--(0,4)--(0,0);
\draw[line width=0.3pt] (0,1)--(6,1) (0,2)--(6,2) (0,3)--(6,3) (1,0)--(1,4) (2,0)--(2,4) (3,0)--(3,4) (4,0)--(4,4) (5,0)--(5,4);
\draw[line width=0.4pt, color=green] (0,0)--(4,4)--(6,2)--(4,0)--(0,4); \draw[line width=0.4pt, color=blue] (6,0)--(2,4)--(0,2)--(2,0)--(6,4);
\draw[line width=0.4pt, color=red] (1,0)--(0,1)--(3,4)--(6,1)--(5,0)--(1,4)--(0,3)--(3,0)--(6,3)--(5,4)--(1,0);
\end{tikzpicture}
\vspace{0.2cm}
\caption{At least $3$ beams of light needed for a $6\times4$ plane grid}
\end{figure}
\end{example}
\vspace{-0.6cm}

Since a plane grid has four vertices, a beam of light that enters the plane grid from one vertex is bound to come out from the other vertex. Therefore, for any plane grid, at least two beams of light are needed, and we call the trajectories of these two beams an open path. Except that these two open paths are both called closed paths. Obviously, we only need to get the number of closed paths, and write it as $C(m,n)$. Example 1.2 shows that $C(6,4)=1$.

For any $m,n\in \mathbb{N}^{*}$, we have the following theorem.
\begin{theorem}
$C(m,n)=\gcd(m,n)-1$ holds for any $m,n\in \mathbb{N}^*$, where $\gcd(m,n)$ is the greatest common factor of $m$ and $n$.
\end{theorem}
We will present three proofs of Theorem 1.3.

\section{The first proof of Theorem 1.3}
The first proof of Theorem 1.3 is based on Euclidean algorithm. We first prove two lemmas.
\begin{lemma}
$C(n,n)=n-1$ holds for any $n\in \mathbb{N}^*$.
\end{lemma}

\begin{proof}
In Figure 3, we see that the light starting from a point on the boundary of the $n\times n$ plane grid returns to the starting point after three reflections.
\vspace{-0.1cm}
\begin{figure}[H]
\begin{tikzpicture}[scale=0.9]
\draw[line width=0.6pt] (0,0)--(3,0)--(3,3)--(0,3)--(0,0); \draw[line width=0.4pt, color=red] (1,0)--(3,2)--(2,3)--(0,1)--(1,0);
\draw (0.5,0) node[below] {$s$}; \draw (2,0) node[below] {$t$}; \draw (3,2.5) node[right] {$s$}; \draw (3,1) node[right] {$t$};
\draw (2.5,3) node[above] {$s$}; \draw (1,3) node[above] {$t$}; \draw (0,0.5) node[left] {$s$}; \draw (0,2) node[left] {$t$};
\end{tikzpicture}
\vspace{0.2cm}
\caption{A closed path in $n\times n$ plane grid}
\end{figure}
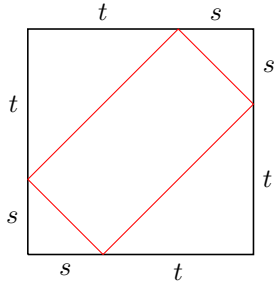
\vspace{-0.6cm}
Specifically, its trajectory is $(s,0)\rightarrow (n,t)\rightarrow (t,n)\rightarrow (0,s) \rightarrow (s,0)$, where $s+t=n$. Therefore, any closed path uniquely corresponds to a point on the boundary that is not a vertex. Hence, the truth of Lemma 2.1 is now clear, because there are $n-1$ points that are not vertices on any boundary of the $n\times n$ plane grid.
\end{proof}

\begin{lemma}
i) $C(m,n)=C(n,m)$ holds for any $m,n\in \mathbb{N}^*$.

ii) $C(m,n)=C(m+n,n)$ holds for any $m,n\in \mathbb{N}^*$.
\end{lemma}

\begin{proof}
The statement i) is obviously valid. In Figure 4, we see that a beam of light entering the $n\times n$ plane grid on the right from point $P$ will pass through point $P$ again and return to its original path. Therefore, when we remove the $n\times n$ plane grid on the right, the number of closed paths of the new plane grid remains unchanged. This proves the statement ii).
\vspace{-0.1cm}
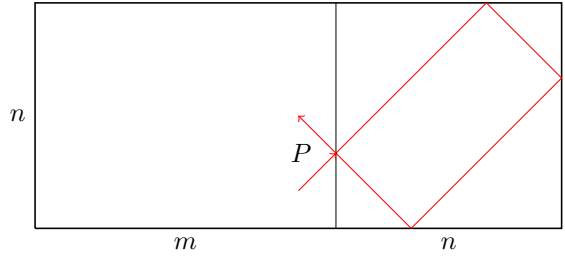
\begin{figure}[H]
\begin{tikzpicture}
\draw[line width=0.6pt] (0,0)--(7,0)--(7,3)--(0,3)--(0,0); \draw[line width=0.4pt, color=red] (5,0)--(7,2)--(6,3)--(4,1)--(5,0);
\draw[line width=0.35pt] (4,0)--(4,3); \draw[line width=0.4pt,color=red,->] (3.5,0.5)--(4,1); \draw[line width=0.4pt,color=red,->] (4,1)--(3.5,1.5);
\draw (2,0) node[below] {$m$}; \draw (5.5,0) node[below] {$n$}; \draw (0,1.5) node[left] {$n$}; \draw (3.8,1) node[left] {$P$};
\end{tikzpicture}
\vspace{0.2cm}
\caption{A closed path in $(m+n)\times n$ plane grid}
\end{figure}
\vspace{-0.6cm}
\end{proof}

\begin{proof}[{\bf\text{The first proof of Theorem 1.3}}]
Based on Lemma 2.1 and Lemma 2.2, Theorem 1.3 can be proved immediately by Euclidean algorithm.
\end{proof}

\section{The second proof of Theorem 1.3}
The second proof of Theorem 1.3 is based on a mapping about points contained in plane grids and the concept of the step length of closed paths.

The set of all points contained in the $m_1\times m_2$ plane grid is
\[\mathcal{S}^2=\{(x_1,x_2)~\vert~0\leq x_i\leq m_i,m_i\in\mathbb{N}^*,i=1,2\}.\]

\begin{definition}
We define a mapping $F$ from the set $\mathcal{S}^2$ to the set $\mathcal{S}^2$ as follows.
\[F:~\mathcal{S}^2\rightarrow \mathcal{S}^2,\quad (x_1,x_2)\mapsto F(x_1,x_2)=(f_1(x_1),f_2(x_2)),\]
where the component $f_i(x_i)$ satisfies the following rule:
\begin{equation}
t=f_i^{2nm_i-x_i-t}(x_i)=f_i^{2nm_i-x_i+t}(x_i),                                                  \label{3-2}
\end{equation}
where $0\leq t\leq m_i,~i=1,2,~n\in\mathbb{Z}$, with $f_i^0(x_i)=x_i$.
\end{definition}

\begin{remark}
The definition of $f_i(x_i)$ is equivalent to the following processes.
\begin{equation}
\begin{split}
t&\mapsto f_i(t)=t+1\mapsto f_i(t+1)=t+2\mapsto\cdots\\
&\mapsto f_i(m_i-1)=m_i\mapsto f_i(m_i)=m_i-1\\
&\mapsto f_i(m_i-1)=m_i-2\mapsto\cdots\mapsto f_i(1)=0\\
&\mapsto f_i(0)=1\mapsto f_i(1)=2\mapsto\cdots,                                                                      \label{3-1}
\end{split}
\end{equation}
and
\begin{equation}
\begin{split}
t&\mapsto f_i^{-1}(t)=t-1\mapsto f^{-1}_i(t-1)=t-2\mapsto\cdots\\
&\mapsto f_i^{-1}(2)=1\mapsto f_i^{-1}(1)=0\mapsto f_i^{-1}(0)=1\\
&\mapsto\cdots\mapsto f_i^{-1}(m_i-1)=m_i\mapsto f_i^{-1}(m_i)=m_i-1\\
&\mapsto f_i^{-1}(m_i-1)=m_i-2\mapsto\cdots,                                                                      \label{3-3}
\end{split}
\end{equation}
where $0\leq t\leq m_i,~i=1,2$.

In fact, we can use a circle to represent the changing law of component $f_i(x_i)$. Take $m_i=5$ as an example.
\vspace{-0.1cm}
\begin{figure}[H]
\begin{tikzpicture}
\draw (0,0) circle (1.5);
\node (P0) at (270:1.75) {$0$}; \node (P0) at (306:1.75) {$1$}; \node (P0) at (342:1.75) {$2$}; \node (P0) at (18:1.75) {$3$}; \node (P0) at (54:1.75) {$4$};
\node (P0) at (90:1.75) {$5$}; \node (P0) at (126:1.75) {$4$}; \node (P0) at (162:1.75) {$3$}; \node (P0) at (198:1.75) {$2$}; \node (P0) at (234:1.75) {$1$};
\end{tikzpicture}
\vspace{0.2cm}
\caption{The changing law of component $f_i(x_i)$ when $m_i=5$.}
\end{figure}
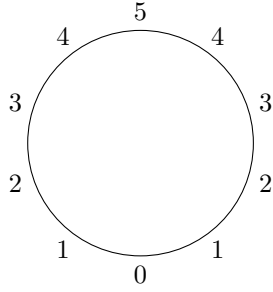
\vspace{-0.6cm}
\end{remark}

It is easy to see that when a beam of light starts from $(x_1,x_2)$, $F(x_1,x_2)$ is the next point that this beam of light passes through. Let $F^k=F\circ F\circ\cdots\circ F$ be the composite of $k$ mappings $F$ with $F^0(x_1,x_2)=(x_1,x_2)$. The chain $F^1,F^2,\cdots,F^k$ is the trajectory of the light starting from point $(x_1,x_2)$.

Before giving the proof of Theorem 1.3, we first study a more general question:
\begin{question}
For any two points $P_1=(x_1,x_2)$ and $P_2=(y_1,y_2)$ in $\mathcal{S}^2$, will the light starting from $P_1$ pass through $P_2$?
\end{question}
This is not always possible for any two points, such as points $P_1=(0,0)$ and $P_2=(0,1)$. But we have the following theorem.

\begin{theorem}
The light starting from $P_1=(x_1,x_2)$ will pass through $P_2=(y_1,y_2)$ if and only if there are $j~(j=0,1)$ and $k~(k=0,1)$ such that
\[2\gcd{(m_1,m_2)}\mid(x_2-x_1)+((-1)^jy_2+(-1)^ky_1).\]
\end{theorem}

\begin{proof}
For any point $P_1=(x_1,x_2)\in\mathcal{S}^2$, if the light starting from $P_1=(x_1,x_2)$ will pass through $P_2=(y_1,y_2)$, there is an integer $k$ such that
\[F^{k}(x_1,x_2)=(f_1^{k}(x_1),f_2^{k}(x_2))=(y_1,y_2).\]
By the definition of $f_i(x_i)$, we let $t=y_i$ in (\ref{3-2}) and get
\[t=y_i=f_i^{2n_im_i-x_i-y_i}(x_i)=f_i^{2n_im_i-x_i+y_i}(x_i),~i=1,2,\]
which is equivalent to
\begin{equation*}
\begin{cases}
\begin{aligned}
&y_1=f_1^{2n_1m_1-x_1-y_1}(x_1)=f_1^{2n_1m_1-x_1+y_1}(x_1),\\
&y_2=f_2^{2n_2m_2-x_2-y_2}(x_2)=f_2^{2n_2m_2-x_2+y_2}(x_2).
\end{aligned}
\end{cases}
\end{equation*}

Thus, we get the following four linear Diophantine equations.
\begin{equation}
\begin{split}
2n_1m_1-x_1-y_1&=2n_2m_2-x_2-y_2;\\
2n_1m_1-x_1-y_1&=2n_2m_2-x_2+y_2;\\
2n_1m_1-x_1+y_1&=2n_2m_2-x_2-y_2;\\                                                       \label{3-4}
2n_1m_1-x_1+y_1&=2n_2m_2-x_2+y_2.
\end{split}
\end{equation}

Eq. (\ref{3-4}) can be rewritten as
\begin{equation}
2m_2n_2-2m_1n_1=(x_2-x_1)+((-1)^jy_2+(-1)^ky_1),                                                       \label{3-5}
\end{equation}
where $j=0,1,~k=0,1$.

Eq. (\ref{3-5}) has integer solutions with respect to $n_1$ and $n_2$ if and only if
\[2\gcd{(m_1,m_2)}\mid(x_2-x_1)+((-1)^jy_2+(-1)^ky_1).\]
\end{proof}

\begin{example}
When $m_1=6,m_2=4$, the light starting from $P_1=(0,3)$ will pass through $P_2=(3,4)$ because
\[2\gcd{(6,4)}=4\mid\pm4=(x_2-x_2)+((-1)^jy_2-y_1),j=0,1.\]
But the light starting from $P_3=(0,2)$ cannot reach $P_2=(3,4)$ because
\[2\gcd{(6,4)}=4\nmid3,-5,9,1=(x_2-x_2)+((-1)^jy_2+(-1)^ky_1),j,k=0,1.\]
\vspace{-0.6cm}
\begin{figure}[H]
\begin{tikzpicture}
\draw[line width=0.6pt] (0,0)--(6,0)--(6,4)--(0,4)--(0,0);
\draw[line width=0.35pt] (0,1)--(6,1) (0,2)--(6,2) (0,3)--(6,3) (1,0)--(1,4) (2,0)--(2,4) (3,0)--(3,4) (4,0)--(4,4) (5,0)--(5,4);
\draw[line width=0.4pt, color=red] (0,3)--(1,4)--(5,0)--(6,1)--(3,4);
\draw[line width=0.4pt, color=blue] (6,0)--(2,4)--(0,2)--(2,0)--(6,4);
\node (P1) at (-0.4,3) {$P_1$};
\node (P2) at (3,4.3) {$P_2$};
\node (P3) at (-0.4,2) {$P_3$};
\end{tikzpicture}
\vspace{0.2cm}
\caption{A $6\times4$ plane grid}
\end{figure}
\end{example}
\vspace{-0.6cm}

Next, we need to give the concept of the closed path in set $\mathcal{S}^2$. Note that if there is a $k\in\mathbb{Z}^+$ such that
\[F^k(x_1,x_2)=(x_1,x_2),\]
then this means that the light starting from $(x_1,x_2)$ returns to point $(x_1,x_2)$. But the trajectory at this time may not necessarily form a closed path.

\begin{example}
When $m_1=4,m_2=3$ and $P=(2,2)$, we have
\[F^{8}(2,2)=(2,2).\]
Clearly, this trajectory is not a closed path.
\begin{figure}[H]
\begin{tikzpicture}
\draw[line width=0.6pt] (0,0)--(4,0)--(4,3)--(0,3)--(0,0);
\draw[line width=0.35pt] (0,1)--(4,1) (0,2)--(4,2) (1,0)--(1,3) (2,0)--(2,3) (3,0)--(3,3);
\draw[line width=0.4pt, color=red] (2,2)--(3,3)--(4,2)--(3,1)--(2,0)--(1,1)--(0,2)--(1,3)--(2,2);
\end{tikzpicture}
\vspace{0.2cm}
\caption{A $4\times3$ plane grid}
\end{figure}
\end{example}
\vspace{-0.6cm}

\begin{definition}
For any point $(x_1,x_2)\in\mathcal{S}^2$, if there is a $k\in\mathbb{Z}^+$ such that
\begin{equation}
\begin{cases}
\begin{aligned}
&F^k(x_1,x_2)=(x_1,x_2),\\
&F^{k-1}(x_1,x_2)=F^{-1}(x_1,x_2),                                                   \label{3-6}
\end{aligned}
\end{cases}
\end{equation}
where $F^{-1}(x_1,x_2)=(f_1^{-1}(x_1),f_2^{-1}(x_2))$, then the chain
\[F^1,F^2,\cdots,F^{k}\]
is called a closed path starting from $(x_1,x_2)$. An open path is defined as a closed path with half repeating itself.
\end{definition}

\begin{remark}
From Theorem 3.4, it is easy to see that the light starting from $(x_1,x_2)$ will definitely return to $(x_1,x_2)$, because
\[2\gcd{(m_1,m_2)}\mid0=(x_2-x_1)+((-1)x_2+x_1).\]
This means that the positive integer $k$ in (\ref{3-6}) must exist. It is easy to see that $k\geq 2$, and $k=2$ if and only if $m_1=m_2=1$.
\end{remark}

\begin{definition}
We call the smallest $k(x_1,x_2)\in\mathbb{Z}^+$ that satisfies
\begin{equation}
\begin{cases}
\begin{aligned}
&F^{k(x_1,x_2)}(x_1,x_2)=(x_1,x_2),\\
&F^{k(x_1,x_2)-1}(x_1,x_2)=F^{-1}(x_1,x_2).                                                   \label{3-7}
\end{aligned}
\end{cases}
\end{equation}
as the step length of this closed path. Correspondingly, the step length of the open path is naturally $\frac{k(x_1,x_2)}{2}$.
\end{definition}

\begin{theorem}
The step length $k(x_1,x_2)$ is an invariant independent of $x_1$ and $x_2$, denoted as $\mathcal{K}$. Specifically,
\[\mathcal{K}=2\text{lcm}(m_1,m_2),\]
where $\text{lcm}(m_1,m_2)$ is the least common multiple of $m_1$ and $m_2$.
\end{theorem}

\begin{proof}
For any point $(x_1,x_2)\in\mathcal{S}^2$, we need
\begin{equation*}
\begin{cases}
\begin{aligned}
&F^{k(x_1,x_2)}(x_1,x_2)=(f_1^{k(x_1,x_2)}(x_1),f_2^{k(x_1,x_2)}(x_2))=(x_1,x_2),\\
&F^{k(x_1,x_2)-1}(x_1,x_2)=(f_1^{k(x_1,x_2)-1}(x_1),f_2^{k(x_1,x_2)-1}(x_2))=(f_1^{-1}(x_1),f_2^{-1}(x_2)).
\end{aligned}
\end{cases}
\end{equation*}

\underline{Case 1}. If $x_i>0,~i=1,2$, then $f^{-1}(x_i)=x_i-1$, we thus have
\[F^{-1}(x_1,x_2)=(f_1^{-1}(x_1),f_2^{-1}(x_2))=(x_1-1,x_2-1).\]
In order to get the step length $k(x_1,x_2)$, we let $t=x_i$ and $t=x_i-1$, respectively, in (\ref{3-2}) and get
\[x_i=f_i^{2n_im_i-2x_i}(x_i)=f_i^{2n_im_i}(x_i),~i=1,2,\]
\[x_i-1=f_i^{2n_im_i-2x_i+1}(x_i)=f_i^{2n_im_i-1}(x_i)=f_i^{-1}(x_i),~i=1,2,\]
which is equivalent to
\begin{equation*}
\begin{cases}
\begin{aligned}
&x_1=f_1^{2n_1m_1-2x_1}(x_1)=f_1^{2n_1m_1}(x_1),\\
&x_2=f_2^{2n_2m_2-2x_2}(x_2)=f_2^{2n_2m_2}(x_2),
\end{aligned}
\end{cases}
\end{equation*}
and
\begin{equation*}
\begin{cases}
\begin{aligned}
&x_1-1=f_1^{2n_1m_1-2x_1+1}(x_1)=f_1^{2n_1m_1-1}(x_1)=f_1^{-1}(x_1),\\
&x_2-1=f_2^{2n_2m_2-2x_2+1}(x_2)=f_2^{2n_2m_2-1}(x_2)=f_2^{-1}(x_2).
\end{aligned}
\end{cases}
\end{equation*}
Hence, we have
\begin{equation}
k(x_1,x_2)=2n_1m_1=2n_2m_2.                                                                                                    \label{3.1}
\end{equation}
Therefore, we need to find $n_1$ and $n_2$ that make $2n_1m_1=2n_2m_2$ the smallest. Let $d=\gcd(m_1,m_2)$ be the greatest common factor of $m_1$ and $m_2$, then $m_1=dm'_1$ and $m_2=dm'_2$ with $\gcd(m'_1,m'_2)=1$. The equation (\ref{3.1}) is equivalent to
\begin{equation*}
\frac{n_2}{n_1}=\frac{m_1}{m_2}=\frac{m'_1}{m'_2},
\end{equation*}
which leads to
\begin{equation}
\begin{split}
n_1&=m'_2=\frac{m_2}{d}=\frac{m_2}{\gcd(m_1,m_2)},\\
n_2&=m'_1=\frac{m_1}{d}=\frac{m_1}{\gcd(m_1,m_2)}.               \label{3.2}
\end{split}
\end{equation}
Therefore, the step length of the closed path is
\[k(x_1,x_2)=2n_1m_1=2n_2m_2=\frac{2m_1m_2}{\gcd(m_1,m_2)}=2\text{lcm}(m_1,m_2),\]
where $\text{lcm}(m_1,m_2)$ is the least common multiple of $m_1$ and $m_2$.

\underline{Case 2}. If there exists $i~(i=1,2)$ such that $x_i=0$, then $f^{-1}(x_i)=1$. Taking $(x_1,x_2)=(0,x_2)$ as an example, and other cases can be discussed similarly. We have
\[F^{-1}(x_1,x_2)=(f_1^{-1}(x_1),f_2^{-1}(x_2))=(1,x_2-1).\]
In order to get the step length $k(x_1,x_2)$, we let $t=0,x_2,1$ and $t=x_2-1$, respectively, in (\ref{3-2}) and get
\begin{equation*}
\begin{cases}
\begin{aligned}
&0=f_1^{2n_1m_1}(x_1),\\
&x_2=f_2^{2n_2m_2-2x_2}(x_2)=f_2^{2n_2m_2}(x_2),
\end{aligned}
\end{cases}
\end{equation*}
and
\begin{equation*}
\begin{cases}
\begin{aligned}
&1=f_1^{2n_1m_1-1}(x_1)=f_1^{2n_1m_1+1}(x_1)=f_1^{-1}(x_1),\\
&x_2-1=f_2^{2n_2m_2-2x_2+1}(x_2)=f_2^{2n_2m_2-1}(x_2)=f_2^{-1}(x_2).
\end{aligned}
\end{cases}
\end{equation*}
Therefore, we also get Eq. (\ref{3.1}) and get the same result.

Obviously, the step length is independent of $x_1$ and $x_2$, so the step length of any closed path is the same, and the step length of open path is half that of closed path.
\end{proof}

\begin{corollary}
An arbitrary closed path has a length of $2\sqrt{2}\text{lcm}(m_1,m_2)$.
\end{corollary}

\begin{proof}[{\bf\text{The second proof of Theorem 1.3}}]
Note that each mapping $F$ corresponds to a diagonal of a minimum unit in the plane grid. A $m_1\times m_2$ plane grid has a total of $m_1m_2$ minimum units and $2m_1m_2$ diagonals. Since any plane grid has only two open paths, we get that the number of closed paths is
\[C(m_1,m_2)=\frac{2m_1m_2-2\text{lcm}(m_1,m_2)}{2\text{lcm}(m_1,m_2)}=\gcd(m_1,m_2)-1.\]
\end{proof}

We are also interested in how many boundary points the closed path will pass through, which corresponds to how many times the light is reflected in $\mathcal{S}^2$.

The set of all points on the boundary of the $m_1\times m_2$ plane grid is
\[\overline{\mathcal{S}}^2=\{(x_1,x_2)\in\mathcal{S}^2~\vert~\exists i,x_i=0~\text{or}~m_i,~i=1,2\}.\]

\begin{corollary}
From (\ref{3.2}), the number of points in $\overline{\mathcal{S}}^2$ that any closed path passes through, that is, the number of boundary points, is
\[\begin{split}
B(m_1,m_2)&=2n_1+2n_2\\
&=2\times\frac{m_2}{\gcd(m_1,m_2)}+2\times\frac{m_1}{\gcd(m_1,m_2)}\\
&=\frac{2(m_1+m_2)}{\gcd(m_1,m_2)}.
\end{split}\]
\end{corollary}

\begin{corollary}
Another interesting fact is that the sum of coordinates of all points (which may have the same points) that any closed path passes through is the same. Specifically,
\[\sum_{k=0}^{\mathcal{K}}f^{k}(x_i)=2n_i\bigg(\sum_{t=0}^{m_i}t\bigg)-n_im_i=n_im_i^2=m_i\text{lcm}(m_1,m_2),\]
and
\[\sum_{k=0}^{\mathcal{K}}\sum_{i=1}^{2}f^{k}(x_i)=(m_1+m_2)\text{lcm}(m_1,m_2).\]
\end{corollary}

In Definition 3.7, we call the chain $F^1,F^2,\cdots,F^{k}$ a closed path starting from $(x_1,x_2)$. Now let us consider the set $\mathcal{F}$ composed of the mappings $F^k$, i.e.,
\[\mathcal{F}=\{F^1,F^2,\cdots,F^k,\cdots\}.\]

\begin{corollary}
By Theorem 3.10, the step length $\mathcal{K}$ of any closed path is an invariant and $\mathcal{K}=2\text{lcm}(m_1,m_2)$, so the set $\mathcal{F}$ is a finite set, and its order is $\mathcal{K}$. Furthermore, the set $\mathcal{F}$ forms an Abelian group generated by $F^1$ under the composition of the mappings $F^k$. In particular,
\[\mathcal{F}\simeq\mathbb{Z}_{\mathcal{K}}.\]
Therefore, any closed path corresponds to a finitely generated Abelian group.
\end{corollary}

Now let us consider another interesting question:
\begin{question}
Given any two points $P_1$ and $P_2$ in the $m_1\times m_2$ plane grid, can we reach $P_2$ from $P_1$ along the diagonal of a minimum unit?
\end{question}

\begin{example}
When $m_1=6,m_2=4$, the point $P_1$ can reach $P_2$ but cannot reach $P_3$.
\begin{figure}[H]
\begin{tikzpicture}
\draw[line width=0.6pt] (0,0)--(6,0)--(6,4)--(0,4)--(0,0);
\draw[line width=0.35pt] (0,1)--(6,1) (0,2)--(6,2) (0,3)--(6,3) (1,0)--(1,4) (2,0)--(2,4) (3,0)--(3,4) (4,0)--(4,4) (5,0)--(5,4);
\draw[line width=0.4pt, color=blue] (0,2)--(1,1)--(2,2)--(4,0);
\node (P1) at (-0.4,2) {$P_1$};
\node (P2) at (4,-0.4) {$P_2$};
\node (P3) at (6.4,1) {$P_3$};
\end{tikzpicture}
\vspace{0.2cm}
\caption{A $6\times4$ plane grid}
\end{figure}
\end{example}
\vspace{-0.6cm}

To solve Question 3.14, we define a mapping $F_{s,t}$ from the set $\mathcal{S}^2$ to the set $\mathcal{S}^2$ as follows.
\[F_{s,t}:~\mathcal{S}^2\rightarrow \mathcal{S}^2,\quad (x_1,x_2)\mapsto F_{s,t}(x_1,x_2)=(f_1^{(-1)^s}(x_1),f_2^{(-1)^t}(x_2)),~s,t=0,1,\]
with
\[F_{s,t}^k(x_1,x_2)=(f_1^{(-1)^sk}(x_1),f_2^{(-1)^tk}(x_2)),~k\in\mathbb{Z},\]
where the component $f_i(x_i)$ satisfies the rule (\ref{3-2}).

Let $\mathcal{G}$ be the set composed of some composites of the mappings $F_{s,t}^k$, i.e.,
\[\mathcal{G}=\{F_{s_1,t_1}^{k_1}\circ\cdots\circ F_{s_r,t_r}^{k_r}~|~s_i,t_i=0,1,k_i\in\mathbb{Z},1\leq i\leq r,r\in\mathbb{N}^*\}.\]

It is easy to verify that the set $\mathcal{G}$ forms an Abelian group under the composition of the mappings $F_{s,t}^k$. Furthermore, $\mathcal{G}$ is also finitely generated by $4$ elements $F_{0,0},F_{0,1}$. Therefore,
\[\mathcal{G}=\{F_{0,0}^{k_1}\circ F_{0,1}^{k_2}~|~k_i\in\mathbb{Z},i=1,2\}.\]
By Theorem 3.10, for any $F\in\mathcal{G}$, we have
\[F^{\mathcal{K}}=F.\]
Hence, $\mathcal{G}$ is a finitely generated torsion abelian group and obviously a finite group. Then
\[\mathcal{G}\simeq\mathbb{Z}_{\mathcal{K}}^2=\mathbb{Z}_{\mathcal{K}}\oplus\mathbb{Z}_{\mathcal{K}}.\]

Suppose that $F\in\mathcal{G}$ and $P\in\mathcal{S}^2$. An operation of $\mathcal{G}$ on $\mathcal{S}^2$ is given by the mapping
\[\mathcal{G}\times\mathcal{S}^2\rightarrow\mathcal{S}^2,\quad (F,P)\mapsto F(P).\]

Note that
\[f_1^{(-1)^s}(x_1)\equiv x_1+1\pmod{2},~~f_2^{(-1)^t}(x_2)\equiv x_2+1\pmod{2},\]
thus
\[f_1^{(-1)^s}(x_1)+f_2^{(-1)^t}(x_2)\equiv x_1+x_2\pmod{2}\]
holds for any $s~(s=0,1)$ and $t~(t=0,1)$. We have

\begin{theorem}
The set $\mathcal{S}^2$ is divided into the following two orbits.
\[\begin{split}
S_e&=\{P=(x_1,x_2)~|~x_1+x_2\equiv0\pmod{2}\},\\
S_o&=\{P=(x_1,x_2)~|~x_1+x_2\equiv1\pmod{2}\},
\end{split}\]
with
\[S_e\cup S_o=\mathcal{S}^2,~~S_e\cap S_o=\varnothing,\]
and
\begin{equation*}
|S_e|=
\begin{cases}
|S_o|+1, &m_1\equiv m_2\equiv 0\pmod{2},\\
|S_o|,   &\text{otherwise}.
\end{cases}
\end{equation*}
\end{theorem}
Now the answer to Question 3.14 is clear. The point $P_1$ can reach $P_2$ if and only if $P_1$ and $P_2$ are in the same orbit.

\section{The generalization of the problem}
In this section, we first extend Questions 1.1 and 3.3 to $p$-dimensional space, where $p\geq2$. The second proof of Theorem 1.3 provides us with a way to define Questions 1.1 and 3.3 abstractly in high-dimensional space.

\begin{definition}
A $m_1\times m_2\times\cdots\times m_p$ spatial grid is defined as the following point set:
\[\mathcal{S}^p=\{(x_1,x_2,\cdots,x_p)~\vert~0\leq x_i\leq m_i,m_i\in\mathbb{N}^*,i=1,2,\cdots,p,~p\geq2\}.\]

The set of vertices of $m_1\times m_2\times\cdots\times m_p$ spatial grid is defined as
\[\hat{\mathcal{S}^p}=\{(x_1,x_2,\cdots,x_p)\in\mathcal{S}^p~\vert~\text{for all}~i,~x_i=0~\text{or}~m_i\}.\]
\end{definition}

\begin{definition}
We define a mapping $F$ from the set $\mathcal{S}^p$ to the set $\mathcal{S}^p$ as follows.
\[F:~\mathcal{S}^p\rightarrow \mathcal{S}^p,\quad (x_1,x_2,\cdots,x_p)\mapsto F(x_1,x_2,\cdots,x_p)=(f_1(x_1),f_2(x_2),\cdots,f_p(x_p)),\]
where the component $f_i(x_i)$ satisfies the following rule:
\begin{equation}
t=f_i^{2nm_i-x_i-t}(x_i)=f_i^{2nm_i-x_i+t}(x_i),                                                      \label{4-1}
\end{equation}
where $0\leq t\leq m_i,~i=1,2,\cdots,p,~n\in\mathbb{Z}$, with $f_i^0(x_i)=x_i$.
\end{definition}
The definition of $f_i(x_i)$ is also equivalent to the following processes (\ref{3-1}) and (\ref{3-3}).

We first consider the generalization of Question 3.3. For any two points $P_1=(x_1,x_2,\cdots,x_p)$ and $P_2=(y_1,y_2,\cdots,y_p)$, we have

\begin{theorem}
The light starting from $P_1=(x_1,x_2,\cdots,x_p)$ will pass through $P_2=(y_1,y_2,\cdots,y_p)$ if and only if there are $k_i~(k_i=0,1),i=1,2,\cdots,p$ such that the following Diophantine equations
\begin{equation}
\begin{split}
2n_1m_1-x_1+(-1)^{k_1}y_1&=2n_2m_2-x_2+(-1)^{k_2}y_2\\
&=\cdots=2n_pm_p-x_p+(-1)^{k_p}y_p                                                                   \label{4-2}
\end{split}
\end{equation}
has an integer solution with respect to $n_i,~i=1,2,\cdots,p$.
\end{theorem}

\begin{proof}
For any point $P_1=(x_1,x_2,\cdots,x_p)\in\mathcal{S}^p$, if the light starting from $P_1$ will pass through $P_2=(y_1,y_2,\cdots,y_p)$, there is an integer $k$ such that
\[F^{k}(x_1,x_2,\cdots,x_p)=(f_1^{k}(x_1),f_2^{k}(x_2),\cdots,f_p^{k}(x_p))=(y_1,y_2,\cdots,y_p).\]
By the definition of $f_i(x_i)$, we let $t=y_i$ in (\ref{4-1}) and get
\[y_i=f_i^{2n_im_i-x_i+(-1)^{k_i}y_i}(x_i),~k_i=0,1,~i=1,2,\cdots,p.\]
Therefore, we get Theorem 4.3.
\end{proof}

\begin{remark}
A necessary condition for Eq. (\ref{4-2}) to have a solution is for any $1\leq i<j\leq p$, there are $k_i=0,1,k_j=0,1$ such that
\begin{equation}
2\gcd{(m_i,m_j)}\mid(x_j-x_i)+((-1)^{k_j}y_j+(-1)^{k_i}y_i).                                                            \label{4-3}
\end{equation}
Especially, when $p=2$, condition (\ref{4-3}) is sufficient, which is Theorem 3.4.
\end{remark}

\begin{definition}
For any point $(x_1,x_2,\cdots,x_p)\in\mathcal{S}^p$, if there is a $k\in\mathbb{Z}^+$ such that
\begin{equation}
\begin{cases}
\begin{aligned}
&F^k(x_1,x_2,\cdots,x_p)=(x_1,x_2,\cdots,x_p),\\
&F^{k-1}(x_1,x_2,\cdots,x_p)=F^{-1}(x_1,x_2,\cdots,x_p),                                                      \label{4-4}
\end{aligned}
\end{cases}
\end{equation}
where $F^{-1}(x_1,x_2,\cdots,x_p)=(f_1^{-1}(x_1),f_2^{-1}(x_2),\cdots,f_p^{-1}(x_p))$, then the chain
\[F^1,F^2,\cdots,F^{k}\]
is called a closed path starting from $(x_1,x_2,\cdots,x_p)$. An open path is defined as a closed path with half repeating itself.
\end{definition}

\begin{remark}
From Theorem 4.3, it is easy to see that the light starting from $(x_1,x_2,\cdots,x_p)$ will definitely return to $(x_1,x_2,\cdots,x_p)$, because when $k_i=0$ for all $i=0,1,\cdots,p$, the following Diophantine equation
\[2n_1m_1=2n_2m_2=\cdots=2n_pm_p\]
has infinitely many integer solutions. This means that the positive integer $k$ in (\ref{4-4}) must exist.
\end{remark}

\begin{definition}
We call the smallest $k(x_1,x_2,\cdots,x_p)\in\mathbb{Z}^+$ that satisfies
\begin{equation}
\begin{cases}
\begin{aligned}
&F^{k(x_1,x_2,\cdots,x_p)}(x_1,x_2,\cdots,x_p)=(x_1,x_2,\cdots,x_p),\\
&F^{k(x_1,x_2,\cdots,x_p)-1}(x_1,x_2,\cdots,x_p)=F^{-1}(x_1,x_2,\cdots,x_p).                                                   \label{4-5}
\end{aligned}
\end{cases}
\end{equation}
as the step length of this closed path. Correspondingly, the step length of the open path is naturally $\frac{k(x_1,x_2,\cdots,x_p)}{2}$.
\end{definition}

\begin{theorem}
The step length $k(x_1,x_2,\cdots,x_p)$ is an invariant independent of $x_1,x_2,\cdots,x_p$, denoted as $\mathcal{K}$. Specifically,
\[\mathcal{K}=2\text{lcm}(m_1,m_2,\cdots,m_p).\]
where $\text{lcm}(m_1,m_2,\cdots,m_p)$ is the least common multiple of $m_1,m_2,\cdots,m_p$.
\end{theorem}

\begin{proof}
For any point $(x_1,x_2,\cdots,x_p)\in\mathcal{S}^p$, we need
\[\begin{split}
&F^{k(x_1,x_2,\cdots,x_p)}(x_1,x_2,\cdots,x_p)=(x_1,x_2,\cdots,x_p)\\
=~&(f_1^{k(x_1,x_2,\cdots,x_p)}(x_1),f_2^{k(x_1,x_2,\cdots,x_p)}(x_2),\cdots,f_p^{k(x_1,x_2,\cdots,x_p)}(x_p)),
\end{split}\]
and
\[\begin{split}
&F^{k(x_1,x_2,\cdots,x_p)-1}(x_1,x_2,\cdots,x_p)=(f_1^{-1}(x_1),f_2^{-1}(x_2),\cdots,f_p^{-1}(x_p))\\
=~&(f_1^{k(x_1,x_2,\cdots,x_p)-1}(x_1),f_2^{k(x_1,x_2,\cdots,x_p)-1}(x_2),\cdots,f_p^{k(x_1,x_2,\cdots,x_p)-1}(x_p)).
\end{split}\]

\underline{Case 1}. If $x_i>0,~1\leq i\leq p$, then $f^{-1}(x_i)=x_i-1$, we thus have
\[\begin{split}
F^{-1}(x_1,x_2,\cdots,x_p)&=(f_1^{-1}(x_1),f_2^{-1}(x_2),\cdots,f_p^{-1}(x_p))\\
&=(x_1-1,x_2-1,\cdots,x_p-1).
\end{split}\]
In order to get the step length $k(x_1,x_2,\cdots,x_p)$, we let $t=x_i$ and $t=x_i-1$, respectively, in (\ref{4-1}) and get
\[x_i=f_i^{2n_im_i-2x_i}(x_i)=f_i^{2n_im_i}(x_i),~1\leq i\leq p,\]
\[x_i-1=f_i^{2n_im_i-2x_i+1}(x_i)=f_i^{2n_im_i-1}(x_i)=f_i^{-1}(x_i),~1\leq i\leq p,\]
Hence, we have
\begin{equation}
k(x_1,x_2,\cdots,x_p)=2n_1m_1=2n_2m_2=\cdots=2n_pm_p.                                                                        \label{4-6}
\end{equation}
Therefore, we need to find $n_i,1\leq i\leq p$ that make
\[2n_1m_1=2n_2m_2=\cdots=2n_pm_p\]
the smallest. Thus, we have
\begin{equation}
n_i=\frac{\text{lcm}(m_1,m_2,\cdots,m_p)}{m_i},                                                                         \label{4.1}
\end{equation}
where $\text{lcm}(m_1,m_2,\cdots,m_p)$ is the least common multiple of $m_i,1\leq i\leq p$.

\underline{Case 2}. If there exists $i~(1\leq i\leq p)$ such that $x_i=0$, then $f^{-1}(x_i)=1$. Taking $(x_1,x_2)=(0,x_2,\cdots,x_p)$ as an example, and other cases can be discussed similarly. We have
\[\begin{split}
F^{-1}(x_1,x_2,\cdots,x_p)&=(f_1^{-1}(x_1),f_2^{-1}(x_2),\cdots,f_p^{-1}(x_p))\\
&=(1,x_2-1,\cdots,x_p-1).
\end{split}\]
In order to get the step length $k(x_1,x_2,\cdots,x_p)$, we let $t=0,x_i,1$ and $t=x_i-1$, respectively, in (\ref{4-1}) and get
\begin{equation*}
\begin{cases}
\begin{aligned}
&0=f_1^{2n_1m_1}(x_1),\\
&x_i=f_i^{2n_im_i-2x_i}(x_i)=f_i^{2n_im_i}(x_i),~2\leq i\leq p,
\end{aligned}
\end{cases}
\end{equation*}
and
\begin{equation*}
\begin{cases}
\begin{aligned}
&1=f_1^{2n_1m_1-1}(x_1)=f_1^{2n_1m_1+1}(x_1)=f_1^{-1}(x_1),\\
&x_i-1=f_i^{2n_im_i-2x_i+1}(x_i)=f_i^{2n_im_i-1}(x_i)=f_i^{-1}(x_i),~2\leq i\leq p.
\end{aligned}
\end{cases}
\end{equation*}
Therefore, we also get Eq. (\ref{4-6}) and get the same result.

Hence, the step length of the closed path is
\[k(x_1,x_2,\cdots,x_p)=2n_1m_1=2n_2m_2=\cdots=2n_pm_p=2\text{lcm}(m_1,m_2,\cdots,m_p).\]
Obviously, the step length is independent of $x_i,1\leq i\leq p$, so the step length of any closed path is the same, and the step length of open path is half that of closed path.
\end{proof}

\begin{corollary}
An arbitrary closed path has a length of $2\sqrt{n}\text{lcm}(m_1,m_2,\cdots,m_p)$.
\end{corollary}

\begin{theorem}
The number of closed paths (not open paths) is
\[C(m_1,m_2,\cdots,m_p)=2^{p-2}\frac{m_1m_2\cdots m_p}{\text{lcm}(m_1,m_2,\cdots,m_p)}-2^{p-2}.\]
\end{theorem}

When $p=2$, we have
\[C(m_1,m_2)=\frac{m_1m_2}{\text{lcm}(m_1,m_2)}-1=\gcd{(m_1,m_2)}-1,\]
which is exactly Theorem 1.3.

\begin{proof}
Note that each mapping $F$ corresponds to a diagonal of a minimum unit in the spatial grid. A $m_1\times m_2\times\cdots\times m_p$ spatial grid has a total of $m_1m_2\cdots m_p$ minimum units and $2^{p-1}m_1m_2\cdots m_p$ diagonals.

The number of points in the set $\hat{\mathcal{S}^p}$ is $2^p$. An open path corresponds to two points in the set $\hat{\mathcal{S}^p}$, so the number of open paths in the spatial grid is $2^{p-1}$.

Thereore the number of closed paths is
\[\begin{aligned}
C(m_1,m_2,\cdots,m_p)=&\frac{2^{p-1}m_1m_2\cdots m_p-2^{p-1}\text{lcm}(m_1,m_2,\cdots,m_p)}{2\text{lcm}(m_1,m_2,\cdots,m_p)}\\
=&~2^{p-2}\frac{m_1m_2\cdots m_p}{\text{lcm}(m_1,m_2,\cdots,m_p)}-2^{p-2}.
\end{aligned}\]
\end{proof}

\begin{example}
When $p=3$ and $m_1=4,m_2=3,m_3=2$, we have
\[C(4,3,2)=2\times\frac{4\times3\times2}{\text{lcm}(4,3,2)}-2=2.\]
\begin{figure}[H]
\begin{tikzpicture}
\draw[line width=0.6pt] (0,0,6)--(8,0,6)--(8,0,0)--(8,4,0)--(8,4,6)--(0,4,6)--(0,4,0)--(8,4,0) (8,0,6)--(8,4,6) (0,0,6)--(0,4,6) (2,0,6)--(2,4,6)--(2,4,0) (4,0,6)--(4,4,6)--(4,4,0) (6,0,6)--(6,4,6)--(6,4,0) (0,4,2)--(8,4,2)--(8,0,2) (0,4,4)--(8,4,4)--(8,0,4) (0,2,6)--(8,2,6)--(8,2,0);
\draw[dashed,line width=0.6pt] (0,0,6)--(0,0,0)--(8,0,0) (0,0,0)--(0,4,0) (2,0,6)--(2,0,0)--(2,4,0) (4,0,6)--(4,0,0)--(4,4,0) (6,0,6)--(6,0,0)--(6,4,0) (0,4,2)--(0,0,2)--(8,0,2) (0,4,4)--(0,0,4)--(8,0,4) (0,2,6)--(0,2,0)--(8,2,0) (2,2,6)--(2,2,0) (4,2,6)--(4,2,0) (6,2,6)--(6,2,0) (2,4,2)--(2,0,2) (2,4,4)--(2,0,4) (4,4,2)--(4,0,2) (4,4,4)--(4,0,4) (6,4,2)--(6,0,2) (6,4,4)--(6,0,4) (0,2,2)--(8,2,2) (0,2,4)--(8,2,4);
\draw[line width=0.4pt, color=red] (2,0,0)--(4,2,2)--(6,4,4)--(8,2,6)--(6,0,4)--(4,2,2)--(2,4,0)--(0,2,2)--(2,0,4)--(4,2,6)--(6,4,4)
--(8,2,2)--(6,0,0)--(4,2,2)--(2,4,4)--(0,2,6)--(2,0,4)--(4,2,2)--(6,4,0)--(8,2,2)--(6,0,4)--(4,2,6)--(2,4,4)--(0,2,2)--(2,0,0);
\draw[line width=0.4pt, color=green] (0,2,0)--(2,4,2)--(4,2,4)--(6,0,6)--(8,2,4)--(6,4,2)--(4,2,0)--(2,0,2)--(0,2,4)--(2,4,6)--(4,2,4)
--(6,0,2)--(4,2,0)--(6,4,2)--(4,2,4)--(2,0,6)--(0,2,4)--(2,4,2)--(4,2,0)--(6,0,2)--(8,2,4)--(6,4,6)--(4,2,4)--(2,0,2)--(0,2,0);
\end{tikzpicture}
\vspace{0.2cm}
\caption{Two closed paths in the $4\times3\times2$ spatial grid}
\end{figure}
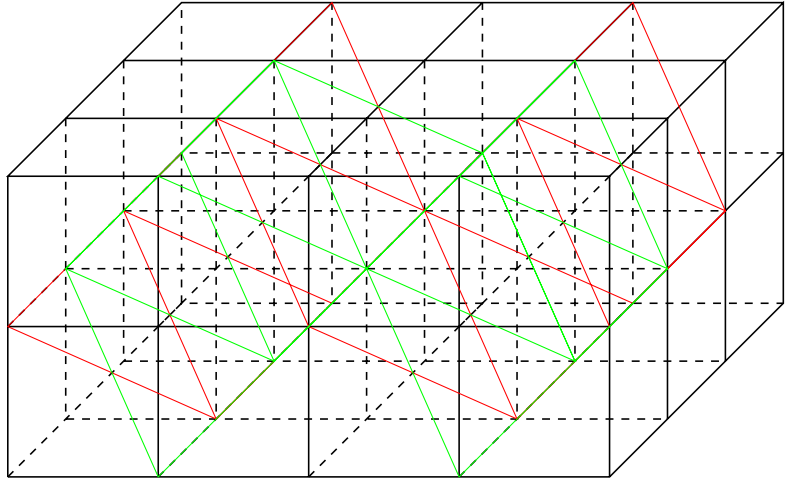
\end{example}
\vspace{-0.6cm}

\begin{corollary}
The sum of coordinates of all points (which may have the same points) that any closed path passes through is the same. Specifically,
\[\sum_{k=0}^{\mathcal{K}}f^{k}(x_i)=2n_i\bigg(\sum_{t=0}^{m_i}t\bigg)-n_im_i=n_im_i^2=m_i\text{lcm}(m_1,m_2,\cdots,m_p),\]
and
\[\sum_{k=0}^{\mathcal{K}}\sum_{i=1}^{p}f^{k}(x_i)=\bigg(\sum_{i=1}^{p}m_i\bigg)\text{lcm}(m_1,m_2,\cdots,m_p).\]
\end{corollary}

In Definition 4.5, we call the chain $F^1,F^2,\cdots,F^{k}$ a closed path starting from $(x_1,x_2,\cdots,x_p)$. Now let us consider the set $\mathcal{F}$ composed of the mappings $F^k$, i.e.,
\[\mathcal{F}=\{F^1,F^2,\cdots,F^k,\cdots\}.\]

\begin{corollary}
By Theorem 4.8, the step length $\mathcal{K}$ of any closed path is an invariant and $\mathcal{K}=2\text{lcm}(m_1,m_2,\cdots,m_p)$, so the set $\mathcal{F}$ is a finite set, and its order is $\mathcal{K}$. Furthermore, the set $\mathcal{F}$ forms an Abelian group generated by $F^1$ under the composition of the mappings $F^k$. In particular,
\[\mathcal{F}\simeq\mathbb{Z}_{\mathcal{K}}.\]
Therefore, any closed path corresponds to a finitely generated Abelian group.
\end{corollary}

We can also consider Question 3.14 in $\mathcal{S}^p$.

\begin{question}
Given any two points $P_1$ and $P_2$ in the $m_1\times m_2\times\cdots\times m_p$ spatial grid, can we reach $P_2$ from $P_1$ along the diagonal of a minimum unit?
\end{question}

\begin{example}
When $m_1=m_2=m_3=1$, the point $P_1$ can reach $P_2$ but cannot reach $P_3$.
\vspace{-0.2cm}
\begin{figure}[H]
\begin{tikzpicture}
\draw[line width=0.6pt] (0,0,2)--(2,0,2)--(2,0,0)--(2,2,0)--(2,2,2)--(0,2,2)--(0,2,0)--(2,2,0) (2,0,2)--(2,2,2) (0,0,2)--(0,2,2);
\draw[dashed,line width=0.6pt] (0,0,2)--(0,0,0)--(2,0,0) (0,0,0)--(0,2,0);
\draw[line width=0.4pt, color=blue] (0,0,2)--(2,2,0);
\node (P1) at (0,0,2.6) {$P_1$};
\node (P2) at (2.4,0,2.8) {$P_3$};
\node (P3) at (2.2,2.2,0) {$P_2$};
\end{tikzpicture}
\vspace{0.2cm}
\caption{A $1\times1\times1$ spatial grid}
\end{figure}
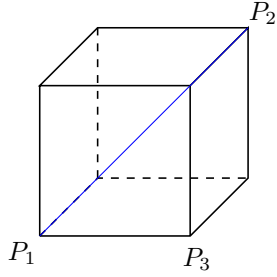
\end{example}
\vspace{-0.6cm}

To solve Question 4.13, we define a mapping $F_{s_1,s_2,\cdots,s_p}$ from the set $\mathcal{S}^p$ to the set $\mathcal{S}^p$ as follows.
\[F_{s_1,s_2,\cdots,s_p}:~\mathcal{S}^p\rightarrow \mathcal{S}^p,\quad (x_1,x_2,\cdots,x_p)\mapsto F_{s_1,s_2,\cdots,s_p}(x_1,x_2,\cdots,x_p),\]
where $s_i=0,1,i=1,2,\cdots,p$ and
\[F_{s_1,s_2,\cdots,s_p}(x_1,x_2,\cdots,x_p)=(f_1^{(-1)^{s_1}}(x_1),f_2^{(-1)^{s_2}}(x_2),\cdots,f_p^{(-1)^{s_p}}(x_p)),\]
with
\[F_{s_1,s_2,\cdots,s_p}^k(x_1,x_2,\cdots,x_p)=(f_1^{(-1)^{s_1}k}(x_1),f_2^{(-1)^{s_2}k}(x_2),\cdots,f_p^{(-1)^{s_p}k}(x_p)),~k\in\mathbb{Z},\]
where the component $f_i(x_i)$ satisfies the rule (\ref{4-1}).

Let $\mathcal{G}$ be the set composed of some composites of the mappings $F_{s_1,s_2,\cdots,s_p}^k$, i.e.,
\[\mathcal{G}=\{F_{s_{11},s_{21},\cdots,s_{p1}}^{k_1}\circ\cdots\circ F_{s_{1r},s_{2r},\cdots,s_{pr}}^{k_r}|s_{ij}=0,1,k_j\in\mathbb{Z},1\leq i\leq p,1\leq j\leq r,r\in\mathbb{N}^*\}.\]

It is easy to verify that the set $\mathcal{G}$ forms an Abelian group under the composition of the mappings $F_{s_1,s_2,\cdots,s_p}^k$. Furthermore, $\mathcal{G}$ is also finitely generated by $2^p$ elements $F_{0,0,\cdots,0},F_{0,0,\cdots,1},\cdots,F_{0,s_{1j},\cdots,s_{pj}},\cdots,F_{0,1,\cdots,1}$. Therefore,
\[\mathcal{G}=\{F_{0,0,\cdots,0}^{k_1}\circ\cdots\circ F_{0,1,\cdots,1}^{k_{2^{p-1}}}~|~k_i\in\mathbb{Z},1\leq i\leq 2^{p-1}\}.\]
By Theorem 4.8, for any $F\in\mathcal{G}$, we have
\[F^{\mathcal{K}}=F.\]
Hence, $\mathcal{G}$ is a finitely generated torsion abelian group and obviously a finite group. Then
\[\mathcal{G}\simeq\mathbb{Z}_{\mathcal{K}}^{2^{p-1}}=\mathbb{Z}_{\mathcal{K}}\oplus\cdots\oplus\mathbb{Z}_{\mathcal{K}}~~(2^{p-1}~\text{summands}).\]

Suppose that $F\in\mathcal{G}$ and $P\in\mathcal{S}^p$. An operation of $\mathcal{G}$ on $\mathcal{S}^p$ is given by the mapping
\[\mathcal{G}\times\mathcal{S}^p\rightarrow\mathcal{S}^p,\quad (F,P)\mapsto F(P).\]

We now prove that the set $\mathcal{S}^p$ is divided into $2^{p-1}$ orbits. By the definition of $f^{(-1)^{s_i}}_i(x_i)$, we have
\[f^{(-1)^{s_i}}_i(x_i)\equiv x_i+1\pmod{2},\]
thus
\begin{equation}
f^{(-1)^{s_i}}_i(x_i)+f^{(-1)^{s_j}}_j(x_j)\equiv x_i+x_j\pmod{2}                                            \label{4-7}
\end{equation}
holds for all $s_i,s_j=0,1,1\leq i,j\leq p$.

For any integer $a$, we introduce the symbol $[a]_2\in\{0,1\}$ such that
\[a\equiv[a]_2\pmod{2}.\]
Using the symbol $[a]_2$, we give the concept of the index of point $P\in\mathcal{S}^p$.

\begin{definition}
The index of point $P=(x_1,x_2,\cdots,x_p)\in\mathcal{S}^p$ is defined as
\[I(P)=([x_1+x_2]_2,[x_1+x_3]_2,\cdots,[x_1+x_p]_2).\]
\end{definition}

\begin{theorem}
The points $P_1=(x_1,x_2,\cdots,x_p)$ and $P_2=(y_1,y_2,\cdots,y_p)$ are in the same orbit if and only if the indexes of $P_1$ and $P_2$ are the same, i.e.,
\begin{equation}
I(P_1)=I(P_2).                                                                                           \label{4-8}
\end{equation}
\end{theorem}

\begin{proof}
If $P_1$ and $P_2$ are in the same orbit, from (\ref{4-7}), we have
\begin{equation}
[x_i+x_j]_2=[y_i+y_j]_2,~1\leq i,j\leq p.                                                                \label{4-9}
\end{equation}
Note that for any $1\leq i,j,k\leq p$, if
\[[x_i+x_j]_2=[y_i+y_j]_2\quad \text{and}\quad [x_i+x_k]_2=[y_i+y_k]_2,\]
then
\[\begin{split}
[x_j+x_k]_2&=[x_j+x_k+2x_i]_2\\
&=[(x_i+x_j)+(x_i+x_k)]_2\\
&=[x_i+x_j]_2+[x_i+x_k]_2\\
&=[y_i+y_j]_2+[y_i+y_k]_2\\
&=[2y_i+y_j+y_k]_2\\
&=[y_j+y_k]_2.
\end{split}\]
Therefore, (\ref{4-8}) and (\ref{4-9}) are equivalent.
\end{proof}

Since $[x_i+x_j]_2\in\{0,1\}$, there are $2^{p-1}$ different $I(P)$.

\begin{theorem}
The set $\mathcal{S}^p$ is divided into $2^{p-1}$ orbits. The points in each orbit have the same index. Denote the orbit with index $I(P)=(\delta_1,\delta_2,\cdots,\delta_{p-1})$ by
\[S_{(\delta_1,\delta_2,\cdots,\delta_{p-1})}=\{P=(x_1,x_2,\cdots,x_p)~|~I(P)=(\delta_1,\delta_2,\cdots,\delta_{p-1})\}.\]
We have
\[\begin{split}
|S_{(\delta_1,\delta_2,\cdots,\delta_{p-1})}|=&~\prod_{i=1}^{p}\frac{(m_i+1)+(-1)^{[\delta_{i-1}]_2}[m_i+1]_2}{2}\\
&+\prod_{i=1}^{p}\frac{(m_i+1)+(-1)^{[\delta_{i-1}+1]_2}[m_i+1]_2}{2},
\end{split}\]
where $\delta_0=0$.
\end{theorem}

When $p=2$, we get the following two orbits
\[\begin{split}
S_{(0)}&=\{P=(x_1,x_2)~|~I(P)=(0)\},\\
S_{(1)}&=\{P=(x_1,x_2)~|~I(P)=(1)\}.
\end{split}\]
This is exactly the orbits $S_e$ and $S_o$ in Theorem 3.16.

Now the answer to Question 4.13 is clear. The point $P_1$ can reach $P_2$ if and only if $P_1$ and $P_2$ are in the same orbit.

\section{The third proof of Theorem 1.3}
The third proof of Theorem 1.3 is obtained on the basis of the second proof by introducing circular sequences and determining the general term formula of circular sequences. At the same time, we consider some combinatorial properties of the mapping $f$ defined in Section 3. The definition of the mapping $f$ is given by the processes (\ref{3-1}) and (\ref{3-3}), rewrite it as
\begin{equation}
\begin{split}
t&\mapsto f(t)=t+1\mapsto f(t+1)=t+2\mapsto\cdots\\
&\mapsto f(m-1)=m\mapsto f(m)=m-1\\
&\mapsto f(m-1)=m-2\mapsto\cdots\mapsto f(1)=0\\
&\mapsto f(0)=1\mapsto f(1)=2\mapsto\cdots,
\end{split}
\end{equation}
and
\begin{equation}
\begin{split}
t&\mapsto f^{-1}(t)=t-1\mapsto f^{-1}(t-1)=t-2\mapsto\cdots\\
&\mapsto f^{-1}(2)=1\mapsto f^{-1}(1)=0\mapsto f^{-1}(0)=1\\
&\mapsto\cdots\mapsto f^{-1}(m-1)=m\mapsto f^{-1}(m)=m-1\\
&\mapsto f^{-1}(m-1)=m-2\mapsto\cdots,
\end{split}
\end{equation}
where $0\leq t\leq m$.

According to the mapping $f$, we give the definition of circular sequences.
\begin{definition}
A positive circular sequence $a^+(n,t,m)$ with the first term $t$ and the height $m$ is defined as
\[a^+(n,t,m)=f^n(t),~~n=0,1,2,\cdots,\]
with $f^0(t)=t$.

A negative circular sequence $a^-(n,t,m)$ with the first term $t$ and the height $m$ is defined as
\[a^-(n,t,m)=f^{-n}(t),~~n=0,1,2,\cdots,\]
with $f^0(t)=t$.
\end{definition}
Obviously, the circular sequence is a periodic sequence with a period of $2m$. When $t=0$ or $t=m$, we have
\[a^+(n,0,m)=a^-(n,0,m)\quad \text{and}\quad a^+(n,m,m)=a^-(n,m,m).\]

From some calculations, we have
\begin{theorem}
\[a^+(n,t,m)=\begin{cases}
n+t,\quad 0\leq n\leq m-t,\\
2m-t-n,\quad m-t+1\leq n\leq 2m-t,\\
n-(2m-t),\quad 2m-t+1\leq n\leq 2m-1,
\end{cases}\]
and
\[a^-(n,t,m)=\begin{cases}
t-n,\quad 0\leq n\leq t-1,\\
n-t,\quad t\leq n\leq m+t,\\
2m+t-n,\quad m+t+1\leq n\leq 2m-1.
\end{cases}\]
Since the period of the circular sequence is $2m$, by induction, we obtain
\[a^+(n,t,m)=\begin{cases}
	n-(2km-t),\quad 2km\leq n\leq (2k+1)m-t,\\
	((2k+2)m-t)-n,\quad (2k+1)m-t+1\leq n\leq (2k+2)m-t,\\
	n-((2k+2)m-t),\quad (2k+2)m-t+1\leq n\leq (2k+2)m-1,
\end{cases}\]
and
\[a^-(n,t,m)=\begin{cases}
	(2km+t)-n,\quad 2km\leq n\leq 2km+t-1,\\
	n-(2km+t),\quad 2km+t\leq n\leq (2k+1)m+t,\\
	((2k+2)m+t)-n,\quad (2k+1)m+t+1\leq n\leq (2k+2)m-1,
\end{cases}\]
where $k$ is a non-negative integer.
\end{theorem}

\begin{example}
When $t=3,m=6$, the positive circular sequence $a^+(n,3,6)$ is given by
\[3,4,5,6,5,4,3,2,1,0,1,2,3,\cdots,\]
and the negative circular sequence $a^-(n,3,6)$ is given by
\[3,2,1,0,1,2,3,4,5,6,5,4,3,\cdots.\]
\end{example}

As the third proof of Theorem 1.3, we only show that Theorems 3.4 and 4.3 can be obtained from circular sequences. We take Theorem 3.4 as an example.

\begin{proof}[Another proof of Theorem 3.4.]
	Without loss of generality, we only consider positive circular sequences. For any point $P_1=(x_1,x_2)\in\mathcal{S}^2$, if the light starting from $P_1=(x_1,x_2)$ will pass through $P_2=(y_1,y_2)$, there is an integer $n$ such that
	\[(y_1,y_2)=(a^+(n,x_1,m_1),a^+(n,x_2,m_2)).\]
Taking $a^+(n,x_1,m_1)=n-(2k_1m_1-x_1)$ and $a^+(n,x_2,m_2)=n-(2k_2m_2-x_2)$ as an example, we have
	\begin{equation*}
		\begin{cases}
			\begin{aligned}
				&y_1=n-(2k_1m_1-x_1),\\
				&y_2=n-(2k_2m_2-x_2).
			\end{aligned}
		\end{cases}
	\end{equation*}
Eliminating $n$, we then obtain
\[2k_1m_1-2k_2m_2=x_1-x_2-y_1+y_2.\]
This is exactly one of the equations in Eq. (\ref{3-4}). For other values of $a^+(n,x_1,m_1)$ and $a^+(n,x_2,m_2)$, we can obtain similar results.
\end{proof}

Now, let us consider the generating functions of the circular sequences.
\[\begin{split}
g^+(x,t,m)=\sum_{n=0}^{\infty}a^+(n,t,m)x^n\quad \text{and}\quad g^-(x,t,m)&=\sum_{n=0}^{\infty}a^-(n,t,m)x^n.
\end{split}\]

\begin{theorem}
Let
\[f^+(x,t,m)=\sum_{n=0}^{2m-1}a^+(n,t,m)x^n\quad \text{and}\quad f^-(x,t,m)=\sum_{n=0}^{2m-1}a^-(n,t,m)x^n,\]
then
\[g^+(x,t,m)=\frac{f^+(x,t,m)}{1-x^{2m}}\quad \text{and}\quad g^-(x,t,m)=\frac{f^-(x,t,m)}{1-x^{2m}}.\]
\end{theorem}

\begin{proof}
We only consider the case of $g^+(t,m,x)$. In fact,
\[\begin{split}
g^+(x,t,m)=&~\sum_{n=0}^{\infty}a^+(n,t,m)x^n\\
=&~\sum_{k=0}^{\infty}a^+(2mk,t,m)x^{2mk}+\sum_{k=0}^{\infty}a^+(2mk+1,t,m)x^{2mk+1}\\
&+\cdots+\sum_{k=0}^{\infty}a^+(2mk+2m-1,t,m)x^{2mk+2m-1}\\
=&~a^+(2mk,t,m)\sum_{k=0}^{\infty}(x^{2m})^k+a^+(2mk+1,t,m)x\sum_{k=0}^{\infty}(x^{2m})^k\\
&+\cdots+a^+(2mk+2m-1,t,m)x^{2m-1}\sum_{k=0}^{\infty}(x^{2m})^k\\
=&~\frac{a^+(0,t,m)+a^+(1,t,m)x+\cdots+a^+(2m-1,t,m)x^{2m-1}}{1-x^{2m}}\\
=&~\frac{\sum\limits_{n=0}^{2m-1}a^+(n,t,m)x^n}{1-x^{2m}}.
\end{split}\]
\end{proof}

In order to get the expressions of $f^+(x,t,m)$ and $f^-(x,t,m)$, we show that the following lemma.
\begin{lemma}
Let
\[F(x,t,n)=tx^{t-1}+(t+1)x^t+\cdots+(n-1)x^{n-2}+nx^{n-1},~~t\geq1,n\geq1,\]
then
\[F(x,t,n)=\frac{tx^{t-1}-(t-1)x^t-(n+1)x^n+nx^{n+1}}{(1-x)^2}.\]

When $t=1$, we have
\[F(x,1,n)=\frac{1-(n+1)x^n+nx^{n+1}}{(1-x)^2}.\]
\end{lemma}

\begin{proof}
\[\begin{split}
\int_0^xF(u,t,n)du&=\int_0^x\left(tu^{t-1}+(t+1)u^t+\cdots+(n-1)u^{n-2}+nu^{n-1}\right)du\\
&=x^t+x^{t+1}+\cdots+x^n\\
&=\frac{x^t(1-x^{n-t+1})}{1-x}.
\end{split}\]
Then
\[F(x,t,n)=\frac{tx^{t-1}-(t-1)x^t-(n+1)x^n+nx^{n+1}}{(1-x)^2}.\]
\end{proof}

\begin{theorem}
For any $t\geq0,m\geq1$, we have
\[\begin{split}
f^+(x,t,m)=\frac{1-x^m}{1-x}\left(\frac{x(1-x^{m-t})}{1-x}-\frac{x^{m-t+1}(1-x^{t-1})}{1-x}+(t-1)x^m+t\right),
\end{split}\]
and
\[\begin{split}
f^-(x,t,m)=\frac{1-x^m}{1-x}\left(\frac{x^{t+1}(1-x^{m-t})}{1-x}-\frac{x(1-x^{t})}{1-x}+tx^m+t\right).
\end{split}\]
\end{theorem}

\begin{proof}
\[\begin{split}
f^+(x,t,m)=&~t+(t+1)x+(t+2)x^2+\cdots+(m-1)x^{m-t-1}\\
&+mx^{m-t}+(m-1)x^{m-t+1}+\cdots+2x^{2m-t-2}+x^{2m-t-1}\\
&+x^{2m-t+1}+2x^{2m-t+2}+\cdots+(t-2)x^{2m-2}+(t-1)x^{2m-1}\\
=&~\frac{1}{x^{t-1}}\left(tx^{t-1}+(t+1)x^{t}+(t+2)x^{t+1}+\cdots+(m-1)x^{m-2}\right)\\
&+x^{2m-t-1}\left(1+2(x^{-1})+\cdots+(m-1)(x^{-1})^{m-2}+m(x^{-1})^{m-1}\right)\\
&+x^{2m-t+1}\left(1+2x+3x^2+\cdots+(t-2)x^{t-3}+(t-1)x^{t-2}\right)\\
=&~\frac{t+(1-t)x+(m-1)x^{m-t+1}-mx^{m-t}}{(1-x)^2}\\
&+\frac{x^{2m-t+1}-(m+1)x^{m-t+1}+mx^{m-t}}{(1-x)^2}\\
&+\frac{x^{2m-t+1}-tx^{2m}+(t-1)x^{2m+1}}{(1-x)^2}\\
=&~\frac{(t-1)x^{2m+1}-tx^{2m}+2x^{2m-t+1}-2x^{m-t+1}+(1-t)x+t}{(1-x)^2}\\
=&~\frac{1-x^m}{1-x}\left(\frac{x(1-x^{m-t})}{1-x}-\frac{x^{m-t+1}(1-x^{t-1})}{1-x}+(t-1)x^m+t\right).\\
\end{split}\]
\[\begin{split}
f^-(x,t,m)=&~t+(t-1)x+(t-2)x^2+\cdots+2x^{t-2}+x^{t-1}\\
&+x^{t+1}+2x^{t+2}+3x^{t+3}+\cdots+(m-1)x^{m+t-1}\\
&+mx^{m+t}+(m-1)x^{m+t+1}+\cdots+(t+2)x^{2m-2}+(t+1)x^{2m-1}\\
=&~x^{t-1}\left(1+2(x^{-1})+3(x^{-1})^2+\cdots+(t-1)(x^{-1})^{t-2}+t(x^{-1})^{t-1}\right)\\
&+x^{t+1}\left(1+2x+3x^2+\cdots+(m-1)x^{m-2}\right)\\
&+x^{2m+t-1}\left((t+1)(x^{-1})^t+(t+2)(x^{-1})^{t+1}+\cdots+m(x^{-1})^{m-1}\right)\\
=&~\frac{t-(1+t)x+x^{t+1}}{(1-x)^2}\\
&+\frac{x^{t+1}-mx^{m+t}+(m-1)x^{m+t+1}}{(1-x)^2}\\
&+\frac{mx^{m+t}-(m+1)x^{m+t+1}-tx^{2m}+(t+1)x^{2m+1}}{(1-x)^2}\\
=&~\frac{(t+1)x^{2m+1}-tx^{2m}-2x^{m+t+1}+2x^{t+1}-(1+t)x+t}{(1-x)^2}\\
=&~\frac{1-x^m}{1-x}\left(\frac{x^{t+1}(1-x^{m-t})}{1-x}-\frac{x(1-x^{t})}{1-x}+tx^m+t\right).
\end{split}\]
\end{proof}

\begin{example}
If $m=4$, we have
\[\begin{split}
&f^+(x,0,4)=x(x+1)^2(x^2+1)^2,\\
&f^+(x,1,4)=(x+1)^2(x^2+1)^2,\\
&f^+(x,2,4)=(x+1)(x^2+1)(x^4-x^3+x^2+x+2),\\
&f^+(x,3,4)=(x+1)(x^2+1)(2x^4-x^3-x^2+x+3),\\
&f^+(x,4,4)=(x+1)(x^2+1)(3x^4-x^3-x^2-x+4).
\end{split}\]
Also,
\[\begin{split}
&f^-(x,0,4)=x(x+1)^2(x^2+1)^2,\\
&f^-(x,1,4)=(x+1)(x^2+1)(2x^4+x^3+x^2-x+1),\\
&f^-(x,2,4)=(x+1)(x^2+1)(3x^4+x^3-x^2-x+2),\\
&f^-(x,3,4)=(x+1)(x^2+1)(4x^4-x^3-x^2-x+3),\\
&f^-(x,4,4)=(x+1)(x^2+1)(3x^4-x^3-x^2-x+4).
\end{split}\]
\end{example}

\end{document}